\documentclass[10pt]{amsart}

\usepackage[utf8x]{inputenc}
\usepackage{amsfonts}
\usepackage{amsmath}
\usepackage{amssymb}
\usepackage{amsthm}
\usepackage{mathrsfs}
\usepackage{latexsym}
\usepackage{graphicx}
\usepackage{bm}
\usepackage{inputenc}
\usepackage{enumerate}
\usepackage{enumitem}

\usepackage{hyperref}
\usepackage{graphicx}

\DeclareMathOperator{\Nor}{Nor}

\DeclareMathOperator{\Unp}{Unp}
\DeclareMathOperator{\dmn}{dmn}

\swapnumbers

\allowdisplaybreaks

\newtheorem{Lemma}{Lemma}[section]
\newtheorem{Corollary}[Lemma]{Corollary}
\newtheorem{Theorem}[Lemma]{Theorem}
\newtheorem*{Theorem*}{Theorem}

\theoremstyle{definition}

\newtheorem{Definition}[Lemma]{Definition}

\theoremstyle{remark}

\newtheoremstyle{proof*}
{3pt}
{3pt}
{\rmfamily}
{}
{\bfseries}
{.}
{.5em}
{\thmnote{#3}}
\theoremstyle{proof*}
\newtheorem*{proof*}{}

\title{When all parallel sets of a $ C^1 $-hypersurface are nowhere $ C^1 $-regular}
\author{Mario Santilli}

\begin{document}

\maketitle

\begin{abstract}
We prove a necessary and sufficient condition for a $ C^1 $-hypersurface to have all parallel sets nowhere $ C^1 $-regular. As a corollary, we deduce that for a generic $ C^1 $-regular convex body all interior parallel bodies have nowhere $ C^1 $-regular boundaries.
\end{abstract}

\section{Introduction}

It is well known (cf.\ \cite[Theorem 4.8]{Federer1959}) that if $ K \subseteq \mathbf{R}^n $ is a set of positive reach $ \rho $, then for every $ 0 < r < \rho $ the parallel set $ \delta_K^{-1}(r) $ is an embedded $ C^{1} $-hypersurface with Lipschitzian unit-normal (in short, $C^{1,1} $-hypersurface). Here the symbol  $\delta_S $ denotes the distance function from $ S \subseteq \mathbf{R}^n $. In particular, if $ K $ is an arbitrary convex body then all exterior parallel bodies have $ C^{1,1} $-boundary; moreover,  if $ \Sigma $ is a closed embedded $ C^{1,1} $-hypersurface then $ \delta_{\Sigma}^{-1}(r) $ is an embedded $ C^{1,1}$-hypersurface for all positive $ r $ sufficiently small.  

In sharp contrast with this nice picture,  we prove in this short note the following two results. We recall that a set $ S \subseteq \mathbf{R}^n $ is said to be nowhere $ C^1 $-regular if for every open set $ U \subseteq \mathbf{R}^n $ the set $ U \cap S $ fails to be a $ C^1 $-regular embedded hypersurface.
\begin{Theorem*}[\protect{cf.\ Corollary \ref{cor 1}}]
	There exist $ C^{1} $ hypersurfaces $ \Sigma \subseteq \mathbf{R}^n $, given as graphs of  $ C^{1} $ functions $ f : \mathbf{R}^{n-1} \rightarrow \mathbf{R} $, such that $ \delta_{\Sigma}^{-1}(r) $ is nowhere $ C^1 $-regular for all $ r > 0 $.
\end{Theorem*}

In  the next result the word "generic" is understood in the sense of Baire category (cf. \cite[1.8 and 2.7]{Schneiderbook}).

\begin{Theorem*}[\protect{cf.\ Corollary \ref{cor 2}}]
	For a generic $ C^1 $-regular convex body $ K \subseteq \mathbf{R}^n $ all interior parallel bodies have nowhere $ C^1 $-regular boundaries.
\end{Theorem*}

Both these results are corollaries of the main theorem of this note, Theorem \ref{main th}, where we prove for an arbitrary open set $ \Omega \subseteq \mathbf{R}^n $ with $ C^1 $ boundary that the existence of interior parallel sets that are somewhere $ C^1 $-regular is equivalent to the validity of a touching uniform ball condition on some open subset of $ \partial \Omega $. Theorem \ref{main th} generalizes \cite[Theorem 2.8]{Santilli21}. 

Recently, in \cite{lamb2025varepsilonneighbourhoodsplanenowheresmoothboundary} it is proved that there exists  $ E \subseteq \mathbf{R}^2 $ and $ \epsilon > 0 $ such that $ \delta_E^{-1}(\epsilon) $ is nowhere $ C^1 $-regular.

\section{Main results}

\begin{Definition}\label{submanifolds}
We say that a subset $ M \subseteq \mathbf{R}^n $ is a $ C^{1}$-hypersurface
 if and only if for every $ a \in M $ there exists an open subset $ U $ of $ \mathbf{R}^n $, an $ n-1 $ dimensional subspace $ Z $ of $ \mathbf{R}^n $ and a $ C^{1} $-diffeomorphism $ \sigma : U \rightarrow \mathbf{R}^n $ such that 
	\begin{equation*}
		\sigma(U \cap M) = Z \cap \sigma(U).
	\end{equation*}
\end{Definition}

\begin{Definition}
	Let $ K \subseteq \mathbf{R}^n $ be a closed set. 
	
	We define the \emph{distance function from $ K $} by 
	$$ \delta_K(x) = \inf\{|x-a| : a \in K\} \quad \text{for $ x \in \mathbf{R}^n $}\,, $$
	the \emph{nearest point-projection onto $ K $} as the multivalued function $ \bm{\xi}_K: \mathbf{R}^n \rightarrow \bm{2}^{K} $ given by
	\begin{equation*}
		\bm{\xi}_K(x) = K  \cap \big\{a: | x - a| = \bm{\delta}_K(x) \big\} \quad \text{for $ x \in \mathbf{R}^n $\,,} 
	\end{equation*}
and we set
$$ \Unp(K) = \{x \in \mathbf{R}^n \setminus K : \text{$\bm{\xi}_K(x)$ is a singleton}\}\,. $$
\end{Definition}

We introduce now a notion of continuity for multivalued maps.

\begin{Definition}[\protect{cf.~\cite[Definition~2]{Zajicek1983diff}}]
	\label{def:multi-cont}
	Let $X$ and $Y$ be normed vectorspaces and $T$ be a~$Y$-mul\-ti\-val\-ued
	map defined on $X$. We say that \emph{$T$ is weakly continuous
		at $a \in X$} if and only if $T(a) \neq \varnothing $ and for each $\varepsilon > 0$
	there exists $\delta > 0$ such that
	\begin{displaymath}
		T(x) \subseteq T(a) + \bm{U}(0,\epsilon) \quad \text{whenever $x \in X $ and
			 and $|x-a| < \delta$} \,.
	\end{displaymath} 
\end{Definition}

We collect some well known facts for $ \delta_K $ and  $ \bm{\xi}_K $.

\begin{Lemma}\label{lem basic}
	If $ K \subseteq \mathbf{R}^n $ is closed then the following statement holds.
	\begin{enumerate}
		\item If $ x \in \mathbf{R}^n \setminus K $, then $ \delta_K $ is differentiable at $ x $ if and only if $ x \in \Unp(K) $, in which case $ | \nabla \delta_K(x) | = 1 $. In particular, $$ \mathscr{L}^n(\mathbf{R}^n \setminus (K \cup \Unp(K))) =0\,. $$
		\item  $ \bm{\xi}_K $ is weakly continuous at every point of $ \mathbf{R}^n $.
		\item $ \delta_K $ is continuously differentiable on the interior of $ \Unp(K) \setminus K $.
	\end{enumerate} 
	\end{Lemma}

\begin{proof}
	See \cite[Theorem 4.8]{Federer1959}  and \cite[Lemma 2.41]{Kolasinski2023b}.
\end{proof}

\begin{Definition}\label{inner ball}
	Suppose $ \Omega $ is an open subset of $ \mathbf{R}^n $ and $ S \subseteq \partial \Omega $. We say that $ \Omega $ satisfies an \emph{inner uniform ball condition on $ S $} if and only if there exists $ \rho > 0 $ such that each point of $ S $ belongs to the boundary of an open ball  of radius $ \rho $ contained in $ \Omega $.
\end{Definition}

\begin{Theorem}\label{main th}
	Suppose $ \Omega \subseteq \mathbf{R}^n $ is an open subset with $ C^1 $-boundary and define $$ S_r = \Omega \cap \{x : \delta_{\partial \Omega}(x) = r\} \quad \text{ for $ r > 0 $.}$$ 
	
	Then $ \Omega $ satisfies an inner uniform ball condition on a non-empty open subset of $ \partial \Omega $ if and only if there exist $ r\in \mathbf{R}^+ $ and $ U \subseteq \Omega $ open, such that $ S_r \cap U $ is a $ C^1 $-hypersurface.
\end{Theorem}

\begin{proof}
	If $ \Omega $ satisfies an inner uniform ball condition on an open subset of $ \partial \Omega $, then by \cite[Theorem 2.8]{Santilli21} there exists an open set $ U \subseteq \Omega \cap \Unp(\partial \Omega) $. It follows from Lemma \ref{lem basic} that $ \delta_{\partial \Omega}| U $ is continuously differentiable with $ |\nabla \delta_{\partial \Omega}| = 1 $ on $ U $, hence $ S_r \cap U  $ is a $ C^1 $-hypersurface whenever $ S_r \cap U $ is not empty.
	
	We prove now the converse. Suppose $ r \in \mathbf{R}^+  $, $ U \subseteq \Omega $ is  open, $ U \cap S_r $ is a $ C^1 $-hypersurface and $$ \text{$ \nu : U \cap S_r \rightarrow \mathbf{S}^{n-1} $ is a continuous unit-normal.} $$
Firstly, we claim that
\begin{equation}\label{claim 1}
	\bm{\xi}_{\partial \Omega}(x) \subseteq \{x + r \nu(x), x-r\nu(x)\} \quad \text{for $ x \in U \cap S_r $}\,.
\end{equation}
Suppose $ x \in U \cap S_r $. If $ a \in \bm{\xi}_{\partial \Omega}(x) $ and $ \eta = \frac{x-a}{r} $ then we observe that $$ \bm{U}\Big(x-\frac{r}{2}\eta, \frac{r}{2}\Big) \cap S_r = \varnothing \quad \text{and} \quad  \eta \in \Nor(S_r, x). $$ Hence, either $ \eta = \nu(x) $ of $ \eta = - \nu(x) $; in other words, either $ a = x + r \nu(x) $ or $ a = x - r \nu(x) $ and \eqref{claim 1} is proved. 

If $ U \cap S_r \cap \Unp(\partial \Omega) = \varnothing $, then 
$$ \bm{\xi}_{\partial \Omega}(x) = \{x + r \nu(x), x-r\nu(x)\} \quad \text{for $ x \in U \cap S_r $} $$
and we can define $ \xi(x) = x + r \nu(x) $ for  $ x \in U \cap S_r $ to obtain a continuous map $ \xi : U \cap S_r \rightarrow \partial \Omega $. 

On the other hand, if there exists $ x_0 \in U \cap S_r \cap \Unp(\partial \Omega) $ then we assume, without loss of generality, that $ \bm{\xi}_{\partial \Omega}(x_0) = \{x_0 + r \nu(x_0)\}  $. Noting that $ \bm{\xi}_{\partial \Omega} $ is weakly continuous at $ x_0 $ and $ \nu $ is continuous at $ x_0 $, we can choose $ 0 < \delta < \infty $  so that $ \bm{U}(x_0, \delta) \subseteq U $ and
\begin{equation*}
	 \bm{\xi}_{\partial \Omega}(x) \subseteq (x_0 + r \nu(x_0)) + \bm{U}(0, r/2) 
\end{equation*}
and 
\begin{equation*}
| (x- r \nu(x)) - (x_0 - r \nu(x_0))| < \tfrac{r}{2}
\end{equation*}
for every $ x \in \bm{U}(x_0, \delta) \cap S_r $. Hence, if there existed $ x \in \bm{U}(x_0, \delta) \cap S_r $ so that $ x - r \nu(x) \in \bm{\xi}_{\partial \Omega}(x) $, then we would infer that
$$ | (x- r \nu(x)) - (x_0 + r \nu(x_0)) | < \tfrac{r}{2} $$ 
and 
$$ 2r = | (x_0 + r \nu(x_0)) - (x_0- r \nu(x_0)) | < r\,, $$
leading to a contradiction. We conclude that
\begin{equation*}\label{claim 3}
 \bm{\xi}_{\partial \Omega}(x) = \{x + r \nu(x)\} \quad \text{for $ x \in \bm{U}(x_0, \delta) \cap S_r $\,,}
\end{equation*}
hence the map $ \xi(x) = x + r \nu(x) $ for $ x \in \bm{U}(x_0, \delta) \cap S_r $ defines a continuous map $ \xi : \bm{U}(x_0, \delta) \cap S_r \rightarrow \partial \Omega $.

In both cases considered above, we see that if $ x \in \dmn \xi $ then $ | x - \xi(x)| = r $ and 
$$ \bm{U}\bigl(\xi(x) + r \tfrac{x-\xi(x)}{r}, r\bigr) \cap  \partial \Omega = \varnothing\,. $$
Hence, $ x- \xi(x) \in \Nor(\partial \Omega, \xi(x)) $ and $ \Omega $ satisfies an inner uniform ball condition on $ \xi(\dmn \xi) $. To conclude the proof it remains to show that $ \xi(\dmn\xi) $ has non empty interior in $ \partial \Omega $. To this end, we choose $ z \in \dmn \xi $ and $ 0 < s < r $ so that $ \bm{U}(z, s) \cap S_r   \subseteq \dmn \xi $. If $  x, y \in  \dmn \xi $ satisfy $\xi(x) = \xi(y) $, we recall that  $ \dim \Nor(\partial \Omega,\xi(x) ) = 1 $, $ | x- \xi(x) | = r = | y - \xi(x) | $ and
\begin{equation*}
y- \xi(x)\,,	x- \xi(x) \in \Nor(\partial \Omega, \xi(x)) \,,
\end{equation*}
to see that 
$$ \text{either $x- \xi(x) =  y- \xi(x) $, or $x- \xi(x) =  \xi(x) - y $.}$$ However, the latter would imply that 
\begin{equation*}
	| x-y | = |x- \xi(x) + \xi(x) - y| = 2 |x-\xi(x)| = 2r
\end{equation*}
which is clearly impossible, since $|x-y| < 2s < 2r $. Henceforth $ x= y $, and we conclude that $ \xi| S_r \cap \bm{U}(z, s) $  is injective.  Since $ \bm{U}(z, s) \cap S_r $ and $ \partial \Omega $ are $ n $-dimensional manifolds, we can now apply the invariance of domain theorem (cf.\ \cite[IV.\ 7.4]{Dold}) to conclude that $ \xi( \bm{U}(z, s) \cap S_r) $ is an open subset of $ \partial \Omega $. 
\end{proof}

We can now easily prove the two main results mentioned in the introduction.

\begin{Corollary}\label{cor 1}
	There exist $ C^{1} $ hypersurfaces $ \Sigma \subseteq \mathbf{R}^n $, given as graphs of  $ C^{1} $ functions $ f : \mathbf{R}^{n-1} \rightarrow \mathbf{R} $, such that $ U \cap \delta_{\Sigma}^{-1}(r) $ fails to be a $ C^1 $-hypersurface whenever $ U \subseteq \mathbf{R}^n $ is open and $ r > 0 $.
\end{Corollary}

\begin{proof}
An example of a complete $ C^{1} $-graph $ \Sigma $ is given in \cite{Kohn1977} such that the two connected components of $ \mathbf{R}^n \setminus \Sigma $ fail to satisfy an inner uniform ball condition on each open subset of $ \Sigma $ (cf.\ proof of \cite[Corollary 2.9]{Santilli21}). Hence the conclusion follows from Theorem \ref{main th}.
\end{proof}

\begin{Corollary}\label{cor 2}
	For a generic $ C^1 $-regular convex body $ K \subseteq \mathbf{R}^n $ the set $  U \cap \delta_{\mathbf{R}^n \setminus K}^{-1}(r) $ fails to be a $ C^1 $-hypersurface whenever $ U \subseteq \mathbf{R}^n $ is open and  $ r > 0 $.
\end{Corollary}
\begin{proof}
By \cite{Schneider79} a generic $ C^1 $ regular convex body $ K $ fails to satisfy an inner uniform ball condition on each open subset of $ \partial K $. Hence the conclusion follows from Theorem \ref{main th}.
\end{proof}


\medskip

\end{document}